\documentclass{article}

\usepackage{microtype}
\usepackage{graphicx}
\usepackage{subfigure}
\usepackage{booktabs} 

\usepackage{hyperref}
\usepackage{algorithm, algorithmic}

\usepackage{amssymb, amsmath, amsthm}
\usepackage{shortcuts}
\usepackage{cleveref}
\usepackage[square,sort,comma,numbers]{natbib}
\usepackage{float}

\graphicspath{{images/},{prebuiltimages/}}

\author{Eugene Ndiaye\footnote{Georgia Institute of Technology (ISyE). Correspondence to: endiaye3@gatech.edu} \qquad Ichiro Takeuchi\footnote{Riken AIP and Nagoya Institute of Technology.}}
\title{Continuation Path with Linear Convergence Rate}

\begin{document}
\date{}
\maketitle

\begin{abstract}
Path-following algorithms are frequently used in composite optimization problems where a series of subproblems, with varying regularization hyperparameters, are solved sequentially. By reusing the previous solutions as initialization, better convergence speeds have been observed numerically. This makes it a rather useful heuristic to speed up the execution of optimization algorithms in machine learning. We present a primal dual analysis of the path-following algorithm and explore how to design its hyperparameters as well as determining how accurately each subproblem should be solved to guarantee a linear convergence rate on a target problem. Furthermore, considering optimization with a sparsity-inducing penalty, we analyze the change of the active sets with respect to the regularization parameter. The latter can then be adaptively calibrated to finely determine the number of features that will be selected along the solution path. This leads to simple heuristics for calibrating hyperparameters of active set approaches to reduce their complexity and improve their execution time.
\end{abstract}

\section{Introduction}

To reduce the computational complexity when solving a parsimonious statistical learning task, some solvers use progressive estimation of the most important variables. This prioritization makes it possible to reduce the dimensionality of the problem and to allocate a greater calculation effort to the most important parts. The benefits are essentially savings in memory and algorithm execution time. As main example, let us consider the least squares with $\ell_1$ regularization commonly known as Lasso \citep{Tibshirani96, Chen_Donoho95}. Denoting $y \in \bbR^n$ the observation vector and $X = [X_1, \cdots, X_p] \in \bbR^{n \times p}$ the matrix of explanatory variables, the Lasso estimator is a solution of the following optimization problem:
\begin{equation}\label{eq:lasso_formulation}
\tbeta{\lambda} \in \argmin_{\beta \in \bbR^p} \frac{1}{2}\norm{y - X\beta}_{2}^{2} + \lambda \norm{\beta}_1 =: P_{\lambda}(\beta) \enspace,
\end{equation}
where the regularization parameter $\lambda > 0$, calibrates the trade-off between data fitting and sparsity of the solutions. Under technical assumptions on the data generation process, it has strong theoretical guarantees on its performance \citep{Buhlmann_VanDeGeer11}. A particularly interesting property of the Lasso estimator is that some of its components are exactly equal to zero depending on the value of $\lambda$. Those coordinates that are equal to zero correspond to the columns of the matrix $X$ that can be ignored in the prediction of $y$. So a natural idea is to detect these variables and eliminate them as early as possible in the resolution of problem~\eqref{eq:lasso_formulation}. \\

Early strategies to effectively solve the problem~\eqref{eq:lasso_formulation} are based on homotopy continuation strategies. Noting that the mapping $\lambda \mapsto \tbeta{\lambda}$ is piecewise linear with explicitly computable breakdown points, exact solution of the Lasso can be computed with parametric programming \eg \texttt{Lars-Lasso} type algorithms \citep{Osborne_Presnell_Turlach00a, Osborne_Presnell_Turlach00b, Efron_Hastie_Johnstone_Tibshirani04, Rosset_Zhu07}. While appealing, the worst case complexity of such methods is exponential in the number of features $p$ \citep{Mairal_Yu12, Gartner_Jaggi_Maria12}, and suffer from several numerical instabilities which make them unpractical in large scale scenarios. Nevertheless, by renouncing the exact calculation, approximate continuation strategies prove to be particularly effective \citep{Friedman_Hastie_Hofling_Tibshirani07}. Instead of directly solving \Cref{eq:lasso_formulation} at $\lambda$ from scratch, they solve a sequence of Lasso problem with a gradually decreasing parameters $\{\lambda_0 > \cdots > \lambda_t> \cdots > \lambda_T = \lambda\}$ starting from a large regularizer $\lambda_0$. Each subproblem is solved approximately \eg using first order algorithm, and initialized with the previous solution.
Under noisy linear model, restrictive condition on the design matrix $X$ and where $\lambda$ is larger than the noise level, it has been shown in \citep{Xiao_Zhang13, Zhao_Liu_Zhang18}, that the Lasso objective function is strongly convex along the solution path and then enjoys a global geometric rate of convergence when the subproblems are solved with proximal gradient descent. However, such theoretical analysis leverage oracle information from distribution of the data which is unknown. The current implementation chooses an \textit{arbitrary} grid sequentially defined as $\lambda_{t+1} = \lambda_t \times s$ for a scalar $s \in (0, 1)$ and without proper early stopping rule that precisely control the optimization error. When it comes to solve \Cref{eq:lasso_formulation} at a prescribed $\lambda$, the actual continuation strategies still leave several open questions. It is unclear how to properly design a sequence of hyperparameters exhibiting desirable optimization properties and how to setup an early stopping condition on the intermediate problems.
One of the main challenge is to understand how the intermediate solutions are actually minimizing the target objective~\eqref{eq:lasso_formulation} \ie how the sequence of optimization errors $\{P_{\lambda}(\tbeta{\lambda_t}) - P_{\lambda}(\tbeta{\lambda})\}_{t \in \bbN}$ decrease with respect to $\lambda_t$ and how the sequence of supports $\{j \in [p]:\, \tbeta{\lambda_t}_j \neq 0 \}_{t \in \bbN}$ can be controlled in order to limit the computational complexity.

\paragraph{Contributions.}
We provide an additional theoretical analysis of approximate continuation path for convex composite optimization problems and its impact on sparse Lasso regularization. Based on smooth and strongly convex regularity assumptions on the loss function, we introduce a stepwise descent lemma along the regularization path that suggests a stopping criterion to avoid unnecessary computations and can guarantee a desired amount of progress toward convergence on the target problem. We obtain an approximate continuation path such that the sequence of optimization errors $\{P_{\lambda}(\tbeta{\lambda_t}) - P_{\lambda}(\tbeta{\lambda})\}_t$ is guaranteed to decrease at a geometric rate. Our result holds for a general convex regularizer, any converging iterative algorithm and does not involves any unknown statistical information. Furthermore, connecting with the sparsity of the Lasso and the safe screening rule literature, we highlight how the regularization parameter can be sequentially decreased while maintaining a desired number of active variables along the path. Both of these analysis lead to specific geometric grids. Thus, they can stand as a theoretical justification of the current practices based on a default geometric grid adopted in popular package like \texttt{scikit-learn} and $\texttt{glmnet}$. These strategies are accompanied with a natural working set heuristics, whose subproblem sizes are tightly related to the generated path. It efficiently solves sparse optimization problem in large scale settings by taking benefits of both feature prioritizations and pathwise optimization with a well designed early stopping criterion.

\paragraph{Notation.}
For a non zero integer $n$, we denote $[n]$ to be the set $\{1, \cdots, n\}$. Given a proper, closed and convex function $f: \bbR^n \to \bbR \cup \{+\infty\}$, we denote $\dom f = \{x \in \bbR^n: f(x) < +\infty\}$. Its Fenchel-Legendre conjugate is $f^*:\bbR^n \to \bbR \cup \{+\infty\}$ defined by $f^*(x^*) = \sup_{x \in \dom f} \langle x^* , x \rangle - f(x)$.
\paragraph{Regularity assumption.}
A function $f$ is $\nu$-smooth and $\mu$-strongly convex if for any $z$ and $z_0$ in $\dom f$, $L_f(z, z_0):= f(z) - f(z_0) - \langle\nabla f(z_0), z - z_0\rangle$ satisfies
\begin{equation*}\mathrm{(A)} \qquad
\frac{\mu}{2}\norm{z - z_0}^2 \leq L_f(z, z_0) \leq  \frac{\nu}{2}\norm{z - z_0}^2 \enspace.
\end{equation*}

\section{Problem Setup}

%
Given a target regularization parameter $\lambda > 0$ and a tolerance level $\epsilon>0$, an approximate continuation path determines a sequence
\begin{align*}
 \lambda_0 > \lambda_{1} > \lambda_{2} > \cdots > \lambda_{T} = \lambda \enspace,
\end{align*}
and the corresponding sequence of accuracy levels
\begin{align*}
 \epsilon_{0}, \epsilon_{1}, \epsilon_{2}, \ldots, \epsilon_{T} = \epsilon \enspace,
\end{align*}
where $\lambda_0$ is a large regularization parameter, $T$ is the size of the path, and $\epsilon_{t}$, $t=0, \ldots, T$ is the parameter which indicates how accurately one should solve the problem at the corresponding $\lambda_{t}$. In an approximate continuation path method, starting from an initial estimate $\beta^{(\lambda_0)}$ at $\lambda_0$, the sequence of approximate solutions $\beta^{(\lambda_{0})}, \ldots, \beta^{(\lambda_{T})}$ are computed up to the accuracy $\epsilon_{t}$ by using the previous approximate solution as a warm-start initialization. The goal is to design the sequences of regularization parameters $\{\lambda_t\}_{t=0}^T$ and accuracy parameters $\{\epsilon_t\}_{t=0}^T$ such that the total computational cost to obtain the optimal target solution $\hat{\beta}^{(\lambda)}$ is minimized.
Despite its popular use, the design of existing continuation strategies remains vaguely motivated and a principled way to determine $\{\lambda_t\}_{t=0}^T$ and $\{\epsilon_t\}_{t=0}^T$ is still lacking.
Let us consider a more general convex composite optimization problems involving a sum of a data fitting function plus a regularization term that enforces specific regularity structure: 
\begin{align}\label{eq:primal}
\tbeta{\lambda} \in \argmin_{\beta \in \bbR^p} f(X\beta) + \lambda \Omega(\beta) = P_{\lambda}(\beta).
\end{align}
The Lasso in \Cref{eq:lasso_formulation} is a particular example where $f(\cdot) = \norm{y - \cdot}^2/2$ and $\Omega$ is the $\ell_1$ norm encouraging sparsity in the solution depending on the parameter $\lambda$; the larger it is, the sparser the solution.

The dual problem associated to \Cref{eq:primal} is formulated as
\begin{align}\label{eq:dual}
\ttheta{\lambda} \in \argmax_{\theta \in \bbR^n} -f^*(-\lambda \theta) - \lambda \Omega^*(X^\top\theta) = D_{\lambda}(\theta) \enspace.
\end{align}
%
%
For the Lasso, $f^*(\cdot) = (\norm{y}^2 - \norm{y + \cdot}^2)/2$ and $\Omega^*(X^\top\theta)$ is equal to zero if for any $j$ in $[p]$, $|X_{j}^{\top}\theta| \leq 1$ and infinite otherwise.
If there is no ambiguity, we will denote $\hat\beta_t = \tbeta{\lambda_t}$ and $\hat\theta_t = \ttheta{\lambda_t}$. \\

Given $(\beta, \theta) \in \dom P_{\lambda} \times \dom D_{\lambda}$, the duality gap is defined as
\begin{equation*}
\Gap_{\lambda}(\beta, \theta) = P_{\lambda}(\beta) - D_{\lambda}(\theta) \enspace.
\end{equation*}
By weak duality, it bounds the optimization error 
$$P_{\lambda}(\beta_t) - P_{\lambda}(\tbeta{\lambda}) \leq \Gap_{\lambda}(\beta_t, \theta_t) \enspace.$$
For the $\ell_1$ regularization and given a suboptimal primal vector $\beta_t$, one can build a corresponding dual feasible vector by simply rescaling the residual (a formula for general $\Omega$ can be found in \cite{Ndiaye_Fercoq_Salmon20}):
\begin{equation*}
\theta_t = \frac{-\nabla f(X\beta_t)}{\alpha_t} \text{ with } \alpha_t = \max(\lambda_t, \normin{X^\top \nabla f(X\beta_t)}_{\infty}) \enspace.
\end{equation*}
The sequence of early stopping rule can be defined through the duality gap. Its variation \wrt to regularization parameter $\lambda$ will be our key tool for more understanding of approximate continuation techniques.\\

Our core assumption is that the function $\mathbb{R}^n: z \mapsto f(z)$ is smooth and strongly convex\footnote{\textbf{Not} the composite function $\mathbb{R}^p: \beta \mapsto f(X\beta)$ whose strong-convexity depends on the smallest eigenvalue of the matrix $X$.}
For the Lasso, $f(z) := \norm{y - z}^2/2$ and is $1$-strongly convex. Without loss of generality, we will also assume that $\inf_z f(z) = 0$. For example, this can be enforced by replacing $f$ with $f - \inf_z f(z)$ in the definition of the loss function without changing the optimal solutions.

\begin{figure*}
  \centering
  \subfigure{\includegraphics[width=0.49\columnwidth]{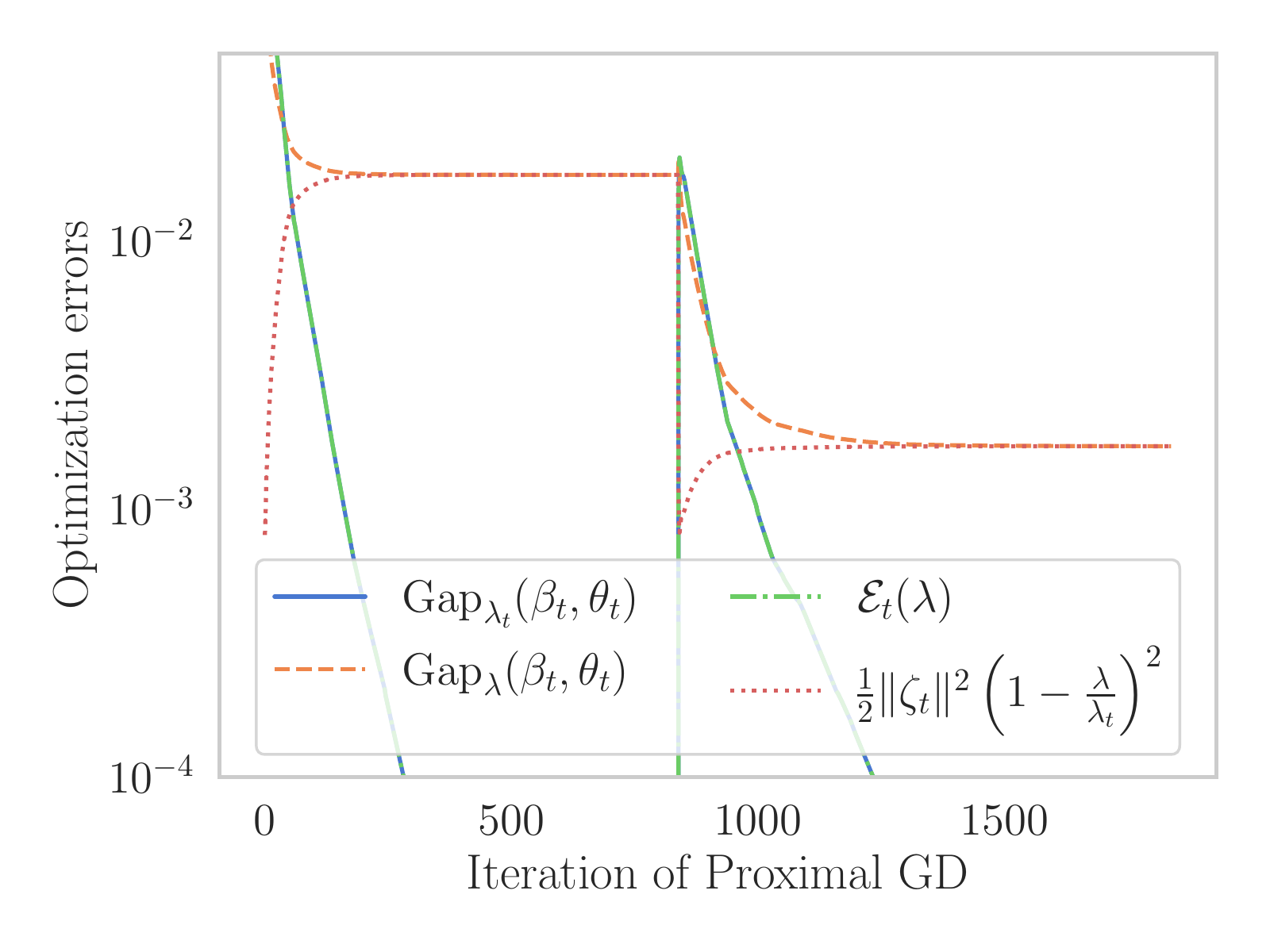}}
  \subfigure{\includegraphics[width=0.49\columnwidth]{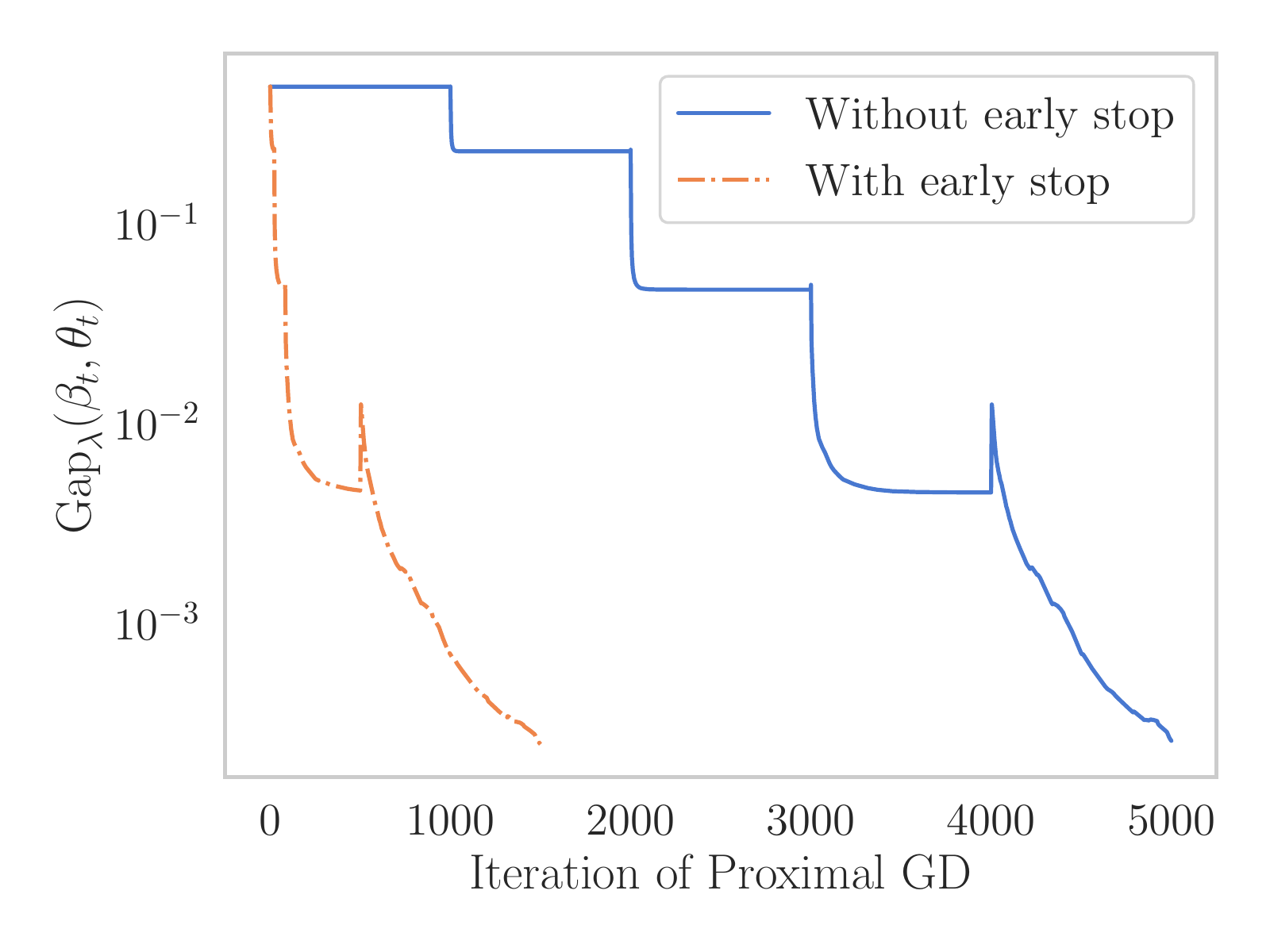}}
  \caption{Evolution of the optimization gaps for a target parameter $\lambda = \frac{\lambda_{\max}}{100}$ with $\lambda_{\max}=\normin{X^\top y}_{\infty}$. On the left, we display the iterates for two intermediate parameters $\lambda_1 = \frac{\lambda_{\max}}{20}$ and $\lambda_2 = \frac{\lambda_{\max}}{50}$. On the right, we use a geometric grid of size $5$ from $\lambda_{\max}$ to $\lambda$.
  The data are synthetically generated as $X \sim \mathcal{N}(0, \mathrm{Id}) \in \bbR^{n \times p}$ with sample size $n=500$ and $p=1000$ features, $y = X\beta^{\star} + \varepsilon$ where $\beta^{\star}$ follows a Laplace distribution with $80\%$ of the coordinates are set to zero and $\varepsilon$ is a Gaussian noise. The optimization algorithm used is a vanilla proximal gradient descent.}
\end{figure*}

\section{Sequential Regularization with Linear Rate}

In this section, we provide a convergence rate analysis of the approximate continuation path framework. To do so, we build on a warm start error bound that quantify the variation of the duality gap of a vector at two different regularization hyperparameter. As a starting point, we have from \citep{Ndiaye_Le_Fercoq_Salmon_Takeuchi2018}:
\begin{align*}
\Gap_{\lambda}(\beta_t, \theta_t) &=
\frac{\lambda}{\lambda_{t}} \Gap_{t} + 
f^*(-\lambda \theta_t) - f^*(-\lambda_{t} \theta_t) +
\left(1 - \frac{\lambda}{\lambda_{t}}\right)[f(X\beta_t)
+ f^*(-\lambda_{t} \theta_t)] \enspace.
\end{align*}
Hence, any bound on the dual variation (with respect to $\lambda$) $f^*(-\lambda \theta_t) - f^*(-\lambda_{t} \theta_t)$ automatically leads to a bound on the duality gap.

\begin{lemma}[\citep{Ndiaye_Le_Fercoq_Salmon_Takeuchi2018}]\label{eq:warmstart_bound}
Let $f$ satisfies assumption $(A)$ and let us define $\zeta_t := -\lambda_t \theta_t$. Then, we have
\begin{align*}
V_\nu(\lambda_t, \lambda) &\leq \Gap_{\lambda}(\beta_t, \theta_t) - \mathcal{E}_t(\lambda) 
\leq V_\mu(\lambda_t, \lambda) \enspace,
\end{align*}
where
\begin{align*}
V_\tau(\lambda_t, \lambda) &= \frac{1}{2\tau} \norm{\zeta_t}^2 \left(1 - \frac{\lambda}{\lambda_t} \right)^2 \,, \text{ for } \tau > 0 \enspace,\\
\mathcal{E}_t(\lambda) &= \frac{\lambda}{\lambda_t}\Gap_t + \left( 1 - \frac{\lambda}{\lambda_t}\right)\Delta_t \enspace,\\
\Delta_{t} &= f(X\beta_t) - f(\nabla f^*(\zeta_{t})) \enspace,\\
\Gap_t &= \Gap_{\lambda_t}(\beta_t, \theta_t) \enspace. 
\end{align*}
\end{lemma}

The \Cref{eq:warmstart_bound} shows that the sequential optimization errors follow an estimation approximation decomposition. The first term $\mathcal{E}_t(\lambda)$ is dominated by the duality gap $\Gap_t$ (which is an upper bound of the optimization error by weak duality) and the rescaling term $\Delta_t$ accounts for a \textit{gap} in the loss function when one travels from the dual to the primal space. One can show that $\Delta_t \leq \norm{\nabla f(X\beta_t)} \sqrt{2\Gap_t / \mu}$ (see \cite[proof of Lemma 3]{Ndiaye_Le_Fercoq_Salmon_Takeuchi2018}). Hence $\mathcal{E}_t(\lambda)$ converges to zero along with the duality gap when the optimization algorithm converges. By definition, $\mathcal{E}_t(\lambda)$ converges to $\Gap_t$ when $\lambda_t$ converges to $\lambda$. When running an optimization algorithm at time $t$, $\lambda_t$ and $\lambda$ are fixed constant, and the convergence of $\mathcal{E}_t(\lambda)$ to zero is solely governed by $\Gap_t$ and $\Delta_t$.
The second term is a quadratic function of the ratio between consecutive hyperparameter that account for the price to pay when one replace $\lambda_t$ by $\lambda$. It is remarkable that when optimality is reached, the targeted optimization error cannot be reduced beyond the limit $\normin{\hat\zeta_t}^2(1 - \lambda / \lambda_t)^2 / 2\nu$.
Therefore, the approximate continuation path can be designed by carefully balancing these two quantities. Building on this result, we introduce a descent lemma on the target problem between two consecutive regularization parameters and we base our design principle on global convergence rate of the pathwise optimization process. 
%
%
\begin{lemma}[Stepwise progress] \label{lm:approximate_stepwise_progress} 
We suppose that the function $f$ satisfies assumption $(A)$ and that the monotonicity condition $f(X\beta_{t+1}) \leq f(X\beta_t)$ holds. Then
\begin{align*}
\Gap_{\lambda}(\beta_{t+1}, \theta_{t+1}) - \Gap_{\lambda}(\beta_{t}, \theta_{t}) \leq \mathcal{E}_{t+1} - \mathcal{E}_t - \frac{\delta_t \normin{\zeta_t}_{2}^{2}}{2 \nu} \enspace,
\end{align*}
where
$$\delta_t := \left(1-\frac{\lambda}{\lambda_{t}}\right)^2 - \left(\frac{\alpha_{t}\nu}{\lambda_{t}\mu}\right)^2 \left(1-\frac{\lambda}{\lambda_{t+1}}\right)^2 \enspace.$$
\end{lemma}

The \Cref{lm:approximate_stepwise_progress} suggests that, at two consecutive steps, if the optimization errors are small enough and the parameter ratio are sufficiently decreasing, then the target duality gap will also decrease. Moreover, it provides an explicit way to design a grid. Sequentially, given $\lambda_{t}$, we can simply choose $\lambda_{t+1}$ to ensure a linear rate of convergence. The latter specifies policy for decreasing the step size parameter $\delta_t$ and the optimization error terms $\mathcal{E}_{t+1} - \mathcal{E}_t$ which characterizes the stopping condition. The \Cref{prop:fast_path} explicitly provides a range of parameters for the adaptive grid and a stopping criterion such that linear convergence is guaranteed.

\begin{proposition}[Linear convergence]\label{prop:fast_path}
Under the assumption of \Cref{lm:approximate_stepwise_progress}, let $r$ in $(0, \frac{\mu}{\nu})$ and let us sequentially define
\begin{align*}
& \lambda_{t+1} = \frac{\lambda}{1 - \sqrt{D(\epsilon_t)}} \text{ and } \mathcal{E}_{t+1} \leq (1 - r) \mathcal{E}_t + \epsilon_t \enspace,
\end{align*}
where $\epsilon_t$ is such that $D(\epsilon_t) \geq 0$ with
$$D(\epsilon) =: \left(\frac{\lambda_{t}\mu}{\alpha_{t}\nu}\right)^2 \left[ \left(1 - r \frac{\nu}{\mu}\right) \left(1 - \frac{\lambda}{\lambda_{t}} \right)^2 -
\frac{2 \nu \epsilon}{\norm{\zeta_t}^2}\right]\enspace.$$ 

Then, we have
\begin{equation*}
\Gap_{\lambda}(\beta_T, \theta_T) \leq (1 - r)^T \Gap_{\lambda}(\beta_0, \theta_0) \enspace.
\end{equation*}
\end{proposition}

This result explains how to choose a sequence of decreasing regularization parameter and optimization stopping criterion to guarantee a desired amount of progress toward convergence. It is solely based on exploiting the regularity of the loss function and does not involve any oracle statistical information that are typically unknown. It holds for any convex regularization function $\Omega$. \\

From \Cref{prop:fast_path}, we can deduce that, if an $\epsilon$-solution is needed at parameter $\lambda$, then it is sufficient to use a geometric grid with
\begin{equation}\label{eq:max_size}
T \leq \frac{ \log\left({\epsilon}/{\Gap_{\lambda}(\beta_0, \theta_0)}\right) }{\log(1 - r)} \enspace.
\end{equation}
This convergence rate reveals that the default parameter of popular solvers ($T=100$) can be unnecessarily large. Similarly, the intermediate optimization error can be unnecessarily small when fixed to the default tolerance $\epsilon$.\\

In our experiences, we observed that $\lambda_t = \lambda$ can be reached exactly which correspond to $D(\epsilon_{t-1}) = 0$. However, it is important to note that the algorithm can output a valid $\epsilon$-solution before reaching $\lambda_t = \lambda$ \ie will stop after a finite number of iterations \Cref{eq:max_size}. In this case, either $\lambda_t = \lambda$ at the last iteration, or an $\epsilon$-solution was found at a $\lambda_t < \lambda$ and the path itself is early stopped. \\

The choice of $\epsilon_t$ is guided by the condition $D(\epsilon_t) \geq 0$ which, for example, trivially holds when one selects $\epsilon_t = 0.42 \times \frac{\norm{\zeta_{t}}^{2}}{2\nu} \left(1 - r \frac{\nu}{\mu} \right) \left(1 - \frac{\lambda}{\lambda_t}\right)^2$. Also the choice of $r = 0.42\frac{\mu}{\nu}$ satisfies the condition $r \in (0, \mu/\nu)$ and the linear convergence is fully guaranteed. Although \Cref{prop:fast_path}, allows us to understand the impact of the choice of such parameters on the convergence speed, the consequences on the wall-clock computation time are not so clear.\\

All the quantities $\mathcal{E}_{t+1}, \mathcal{E}_{t}$ and $\epsilon_t$ displayed in the stopping condition are easily computable and does not add significant cost compared to the classical duality gap evaluation.

\begin{remark}[Prescribed Grid]
Often, a grid of parameters $(\lambda_{t}^{\rm{grid}})_t$ is provided in advance \eg to do cross-validation. In this case, one can use our method by sequentially initializing with $\lambda_0 = \lambda_{t}^{\rm{grid}}$ and $\lambda = \lambda_{t+1}^{\rm{grid}}$ consecutive elements of the given grid. Intermediate points will be added if $\lambda_0$ and $\lambda$ are not close enough. Otherwise, no additional points will be added and only the early stopping rule will be beneficial. We remind that, the target duality gap cannot be reduced with the intermediate grid point after some iterations and must be early stopped. The advantage of our approach here is therefore that it avoids additional calculations that are clearly unnecessary. Another advantage is more noticeable when the regularization induces parsimony, in which case the addition of points in the grid adds flexibility in controlling the size of the supports of optimal solutions; see \Cref{sec:Iterative_Sparse_Optimization}.
\end{remark}

\subsection{Simplified Policies} 
A rigorous application of \Cref{prop:fast_path} may lead to a conservative strategy that is not necessarily efficient or practical in some situations \eg the assumption $(A)$ may not hold and so the conditioning is unknown or the strongly convex constant $\mu$ very small if not equal to zero. For instance, with the logistic regression ($\mu=0$), one need to rely on generalized self-concordance property in order to properly bound the optimization error. However, our analysis provides guidelines for scaling the optimization errors and hyperparameter decrease. Therefore, one can \textit{heuristicaly} use the following simplified policies: set $\lambda_{t+1}$ s.t.
\begin{align*}
&\left(1 - \frac{\lambda}{\lambda_{t+1}} \right)^2 = \left(1 - r\right) \left(1 - \frac{\lambda}{\lambda_{t}} \right)^2 \enspace,\\
%
&\Gap_{t+1} \leq \epsilon_{t+1} := \frac{\lambda_{t+1}}{\lambda} \epsilon \enspace.
\end{align*}
The strategy is to simplify the criterion in \Cref{prop:fast_path} by \textit{arbitrarily} dropping some \textit{nuisance} parameters, which also simplifies the problem.
\begin{itemize}
\item The rescaling term $\Delta_t$ is equal to zero as soon as $\alpha_t = \lambda_t$. This holds whenever the regularization function $\Omega$ is bounded or strongly convex. In these cases, the optimization error term simplifies to $\mathcal{E}_t = \frac{\lambda}{\lambda_t} \Gap_t$. Our analysis suggests a sequential stopping criterion that upper bounds the optimization error by a combination between the previous error and a tolerance $\epsilon_t$ chosen to be roughly smaller than the norm of the residual squared. By setting $\epsilon_t = r\epsilon$, this ensures that $\mathcal{E}_{t+1} \leq \epsilon$ at any time step; from which we deduce the stopping rule above that re-scale the duality gap with respect to the ratio between the current $\lambda_{t+1}$ and the target. Again, far from the target, a high optimization error is not needed for reducing the target duality gap; which is fairly intuitive.

\item Considering the traditional arbitrary geometric grid, one can still use the upper bound in \Cref{eq:max_size} to adaptively design the size of the grid $T_{\max} \propto \log(\epsilon/\Gap_{\lambda}(\beta_0, \theta_0))$. The latter is thus readily calibrated as a function of the quality of the initialization and the target tolerance.

\item One can adaptively choose $r$ as the fraction of the ratio between the target and the current point \ie $r = r_{t+1} \propto \frac{\lambda}{\lambda_t}$ ensuring an increasing sequence of $r_t$ \eg $r_t = \frac{\mu}{\nu} \frac{\lambda}{\lambda_t}$ is a feasible choice. Also note that $\lambda$ and $\epsilon$ are chosen by the algorithm user. When, for some reason, these quantities are too small, one can face numerical issues with overly conservative rules. To alleviate this issue while maintaining sufficient decrease in the approximation path strategy, the scalars $\lambda$ and $\epsilon$ can be respectively replaced by a clipped version \eg $\max(\lambda, \lambda_0 / 10^3)$ and $\max(\epsilon, f(0) / 10^{8})$ for numerical stability.
\end{itemize}

The proposed design of the path with fast global convergence rate holds with any convex regularization $\Omega$ and any converging algorithm for solving the subproblems. Let alone, this might \textit{not} provide a wall clock speed-up of specific optimization solver without exploiting further structures. It rather serves as a generic homotopy continuation policies guided by global convergence analysis and for example can be readily used in interior point methods or sequential smoothing methods. In the following, we apply our analysis in sparse optimization problem as Lasso where one needs to control the computational complexity by limiting the number of active variables which is related to the regularizer $\lambda$.

%
\begin{algorithm}[!t]
\caption{Active Approximate Continuation Path for Sparse Optimization}
\label{alg:active_fast_path}
\begin{algorithmic}
{\small\STATE {\bfseries Input:} Data: $D = X, y$ \qquad Parameters: $\lambda > 0$ \qquad tolerance $\epsilon > 0$
\STATE Initialization: $\beta_0 = 0 \in \bbR^p$ and $\lambda_0 = \norm{X^\top \nabla f(0)}_{\infty}$.
\REPEAT
\STATE \# \textit{Define subproblem}
\STATE $ \lambda_{t+1} = \frac{\lambda}{1 - \sqrt{D(\epsilon_t)}}$ where 
$D(\cdot)$ and $\epsilon_t$ are defined in \Cref{prop:fast_path}.
\STATE
\STATE \# \textit{Solve subproblem with any optional unsafe screening heuristic}
\STATE $\mathcal{W}_{t}(\lambda_t) = \{j \in [p]:\, |X_{j}^{\top} \nabla f(X\beta_{t-1})| \geq \lambda_{t}\}$
\STATE Find $\beta^{(t)} \in \bbR^{p}$ such that $\Gap_{\lambda_{t}}(\beta_{\mathcal{W}_{t}}^{(t)}, \theta^{(t)}) \leq \epsilon_{t}$ \hspace{1cm} \COMMENT{initialized with $\beta_{t-1}$}
\STATE
\STATE \# \textit{Post-processing correction with safe screening rule}
\STATE Find $\beta_{t} \in \bbR^p$ such that $\Gap_{\lambda_{t}}(\beta_{t}, \theta_{t}) \leq \epsilon_{t}$ \hspace{2cm} \COMMENT{initialized with $\beta^{(t)}$}
\STATE Set $(\beta, \theta) = (\beta_{t}, \theta_{t})$
\UNTIL{$\lambda_{t} = \lambda$ or $\Gap_{\lambda}(\beta, \theta) \leq \epsilon$.}
\STATE {\bfseries Return:} $(\beta, \theta)$}
\end{algorithmic}
\end{algorithm}

\begin{figure*}[t!]
  \centering
  \subfigure{\includegraphics[width=\columnwidth]{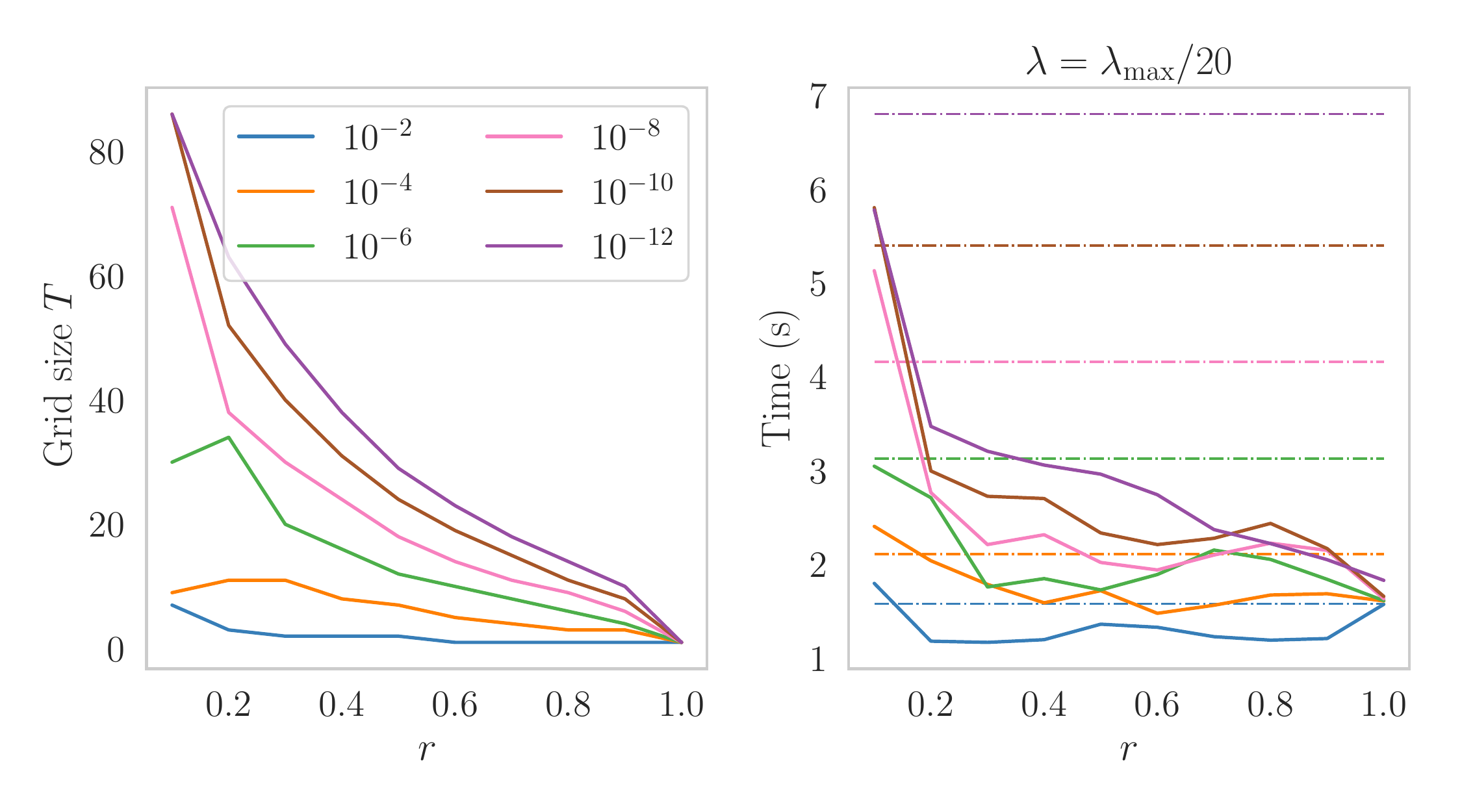}}
  \subfigure{\includegraphics[width=\columnwidth]{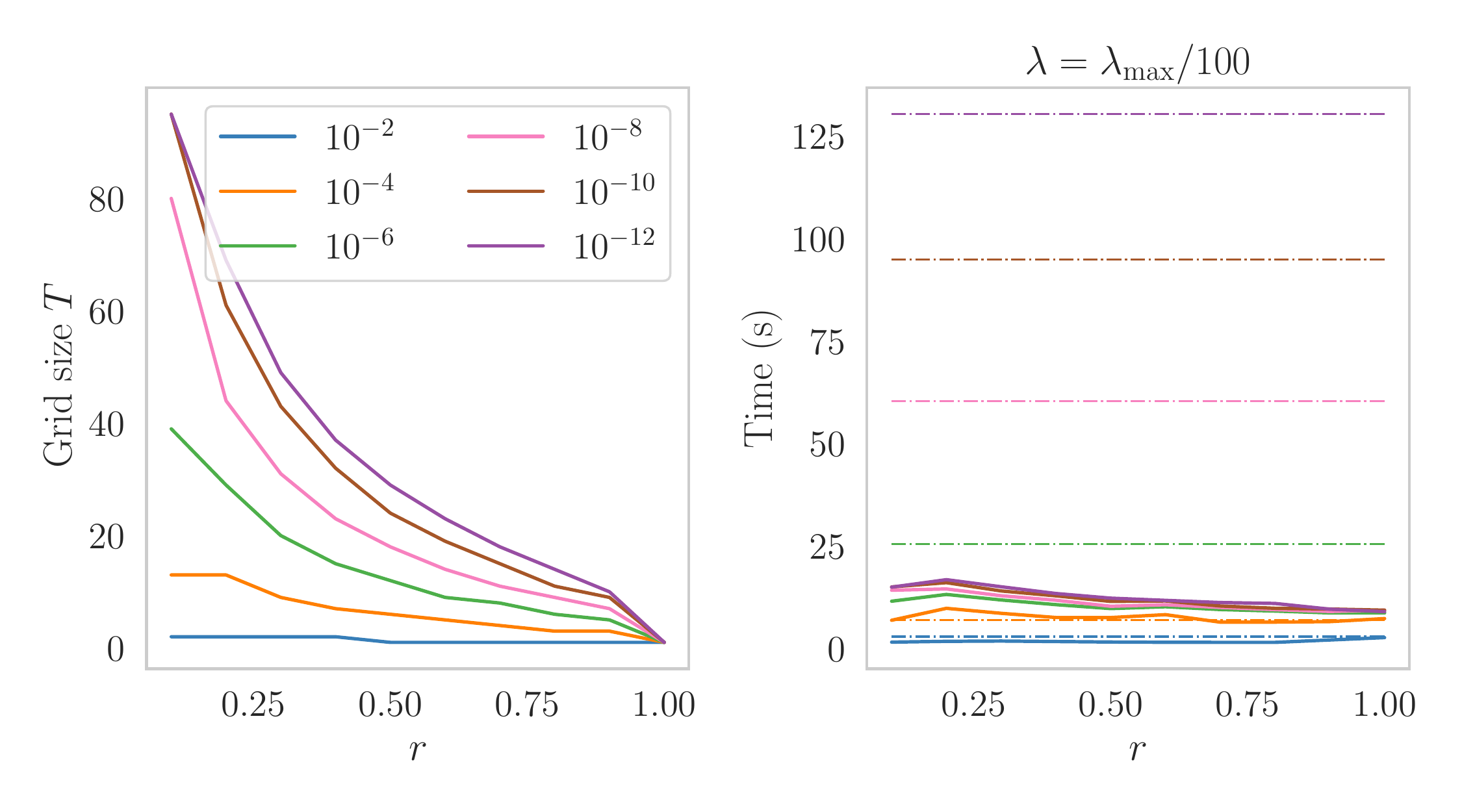}}
  \caption{Illustration of the performances of the proposed approximate continuation path on Lasso for several tolerance error. We use the \texttt{Climate} dataset with $n = 814$ observations and $p = 73577$ features. The optimization problems are solved with the cyclic coordinate descent (CD) with Gap Safe screening Rule for the plain curves and CD without pathwise screening for the dashed line.}
\end{figure*}

\section{Iterative Sparse Optimization}
\label{sec:Iterative_Sparse_Optimization}

We analyze the effect of pathwise optimization on parsimonious problem like Lasso. In the latter, we expect the optimal solution to be less sparse when the regularization parameter is small. We will precisely quantify this fact and derive a regularization path strategy that controls the sparsity level. To simplify the discussion, we limit ourselves to $\ell_1$ regularization but the extensions to more general sparsity inducing regularizer should not be difficult. We refer to \cite{Ndiaye_Fercoq_Salmon20} for the generalization of screening rule for generic separable regularization function $\Omega$.
Let us recall the basic principles of (safe) screening rules introduced in \citep{ElGhaoui_Viallon_Rabbani12} for $\ell_1$ norm and how the duality gap comes into play for eliminating non active variables without false exclusions. Then, we study how its variations \wrt the regularization parameter provide insights on the variations of the active set.

\begin{proposition}[Gap safe rule \cite{Fercoq_Gramfort_Salmon15, Jonhson_Guestrin15}]
\label{prop:gap_safe_screening}
If $f$ is $\nu$-smooth, for any feasible pair $(\beta, \theta) \in \dom P_{\lambda} \times \dom D_{\lambda}$, we have:
\begin{equation*}
\normin{\ttheta{\lambda} - \theta} \leq \sqrt{\frac{2 \nu}{\lambda^{2}} \Gap_{\lambda}(\beta, \theta)} =: r(\lambda) \enspace.
\end{equation*}
Moreover, we have the Gap Safe screening Rule (\texttt{GSR}):
\begin{equation*}
d_{j}(\theta) := \frac{1 - |X_{j}^{\top} \theta|}{\norm{X_{j}}} > r(\lambda) \Longrightarrow \tbeta{\lambda}_j = 0, \qquad \forall j \in [p] \enspace.
\end{equation*}
\end{proposition}

Geometrically, $d_j(\theta)$ is the euclidean distance between $\theta$ and the boundary of the $j$-th constraint in the dual space $\{z \in \bbR^n:\, |X_{j}^{\top}z| \leq 1\}$.
The \Cref{prop:gap_safe_screening} shows that the radius $r$ can be chosen as a function of the duality gap which represents a computable estimation of the optimization error.
To reduce the computational complexity, one can restrict the problem to the active set of variables 
\begin{align*}
\{j \in [p],\, d_j(\theta) \leq r(\lambda) \} &\supset \{j \in [p],\, d_j(\ttheta{\lambda}) = 0 \} 
=: \hat{\mathcal{A}}^{(\lambda)} \enspace.
\end{align*}
Clearly, the ability of such rule to efficiently eliminate many variables depends on how small is the duality gap. In other world, it strongly depends on the quality of the approximate solution $\beta$. \emph{Dynamic screening} rule \citep{Bonnefoy_Emiya_Ralaivola_Gribonval14, Fercoq_Gramfort_Salmon15, Ndiaye_Fercoq_Gramfort_Salmon17} can increase the number of variable eliminated as the algorithm progresses towards the optimal solution which iteratively reduces the approximation error. Still, the latter should be small enough to get an efficient method. In the meantime, the algorithm is running on a large set of features including irrelevant ones. Note that, such safe strategies have to be conservative enough to avoid false exclusion of a non zero coefficient at any optimal solution. There are contexts where solutions to a multitude of hyperparameters chosen on a grid are required \eg cross validation. In that case, a solution $\tbeta{\lambda_{t+1}}$ can be used as an initial estimate of $\tbeta{\lambda_{t}}$ and small approximation error can be expected if the consecutive regularization parameters are close enough. In this setting, at time $t+1$, the information available is the approximate primal/dual solutions at time $t$ \ie $\beta = \beta_t$, $\theta = \theta_t$. We then use the sequential safe radius 
\begin{equation}
r_t(\lambda_{t+1}) = \sqrt{\frac{2 \nu}{\lambda_{t+1}^{2}} \Gap_{\lambda_{t+1}}(\beta_t, \theta_t)}\enspace,
\end{equation}
and the corresponding sequential safe active set:
\begin{equation}
\mathcal{A}_t(\gamma) = \{j \in [p],\, d_j(\theta_t) \leq r_t(\gamma)\} \text{ where } \gamma > 0.
\end{equation}
In the sequential setting, one should pay specific attentions to the path of active sets $\{\mathcal{A}_t(\lambda_{t+1})\}_{t \in [T]}$. Two consecutive safe active sets can be very distinct if their corresponding regularization parameters are far away from each other. To control the differences, the important quantity here is the distance between the dual solutions $\theta_t$ and $\hat\theta_{t+1}$.
From the triangle inequality and \Cref{prop:gap_safe_screening}, we have $\forall j \in [p]$:
\begin{align*}
\left| |X_{j}^{\top} \hat\theta_{t+1}| - |X_{j}^{\top} \theta_{t}| \right| &\leq \norm{X_{j}} \normin{\hat\theta_{t+1} - \theta_t} \\
&\leq \norm{X_{j}} r_t(\lambda_{t+1}) \enspace.
%
\end{align*}
Since the regularization parameter determines the level of sparsity in presence of sparsity inducing regularizer, solving a sequence of optimization problems using a decreasing sequence of hyperparameter resembles active set methods. We will make this explicit in the following.
From the upper bound on the duality gap \Cref{eq:warmstart_bound}, we obtain the following result that explicitly relates the screening abilities to the distance between two consecutive regularization parameters.
\begin{lemma}\label{lm:support_path}
If $f$ satisfies assumption $(A)$, we have
$$d_j(\theta_t) > \sqrt{ \frac{2\nu}{\lambda_{t+1}^{2}} \mathcal{E}_t(\lambda_{t+1})  +  \frac{\nu}{\mu} \norm{\zeta_t}^2 \left( \frac{1}{\lambda_{t+1}} - \frac{1}{\lambda_{t}}  \right)^2 }$$
implies $\tbeta{\lambda_{t+1}}_j = 0$.
\end{lemma}

Then if a coordinate $j$ is eliminated at time $t$ \ie $j \not\in \mathcal{A}^{(\lambda_t)}$, it will also be eliminated at time $t+1$ if $\lambda_{t+1}$ is sufficiently close to $\lambda_t$ and the optimization error term $\mathcal{E}_t(\lambda_{t+1})$ is sufficiently small. Decreasing the regularization parameter will increase the number of safe active variables and so the computational complexity. Beside, an explicit constraint on the optimization error also appears. We must ensure that
$$ \frac{2\nu}{\lambda_{t+1}^{2}} \mathcal{E}_t(\lambda_{t+1}) < d_{j}^{2}(\theta_t),\qquad \forall j \notin \mathcal{A}_t(\lambda_t) \enspace.$$ 
Since $\lambda \leq \lambda_{t+1} \leq \lambda_t$, we have a lower and upper bound of the optimization error term, that is independent of $\lambda_{t+1}$: 
$\mathcal{E}'_{t} \leq \frac{1}{\lambda^{2}} \mathcal{E}_t(\lambda_{t+1}) \leq \mathcal{E}_{t}^{"}$, where
\begin{align*}
\mathcal{E}'_{t} &= \frac{1}{\lambda_{t}^{2}} \Gap_t + \left(\frac{1}{\lambda^{2}} - \frac{1}{\lambda_{t}^{2}} \right) \min(0, \Delta_t) \enspace,\\
\mathcal{E}_{t}^{"} &= \frac{1}{\lambda \lambda_t} \Gap_t + \left(\frac{1}{\lambda^{2}} - \frac{1}{\lambda_{t}^{2}} \right) \max(0, \Delta_t) \enspace.
\end{align*}
The above optimization error constraint is satisfied as soon as the following stopping criterion holds
\begin{equation}
\mathcal{E}_{t}^{"}  < \min_{j \in [p] \backslash \mathcal{A}_t(\lambda_t)} \frac{1}{2 \nu} d_{j}^{2}(\theta_t) \enspace.
\end{equation}

Also, from \Cref{lm:support_path}, one can note that the sequential screening does not improve after the optimization error becomes proportional to the norm of the residual squared. Then, the algorithm can be stopped when
$$\mathcal{E}_t(\lambda_{t+1}) \propto  \frac{1}{2\mu} \norm{\zeta_t}^2 \left(1 - \frac{\lambda_{t+1}}{\lambda_{t}}\right)^2 \enspace.$$

\paragraph{Trade-off between sparsity and fast rate.}
For simplicity, let us consider that we use exact solution at each step. With the proposed sequential step size in \Cref{prop:fast_path} which guarantee linear convergence, and the sequential safe active set control in \Cref{lm:support_path}, we have a correction term proportional to
\begin{equation*}
\sqrt{\frac{\nu}{\mu}} \left( \frac{1}{\lambda_{t+1}} - \frac{1}{\lambda_{t}} \right) = \sqrt{\frac{\mu}{\nu}} \left(\frac{\nu}{\mu} - \sqrt{1 - r \frac{\nu}{\mu}} \right) \left( \frac{1}{\lambda} - \frac{1}{\lambda_{t}} \right) \enspace.
\end{equation*}
So we want $r$ to be as small as possible to reduce the active set error while the global rate of convergence requires $r$ to be as large as possible. In general, we could not confirm an optimal choice of $r$ and this will be considered as a problem dependent hyperparameter. In practice, we simply suggest the previously discussed adaptive choice $r = r_t \propto \frac{\mu}{\nu} \frac{\lambda}{\lambda_t}$. The global rates still hold with the worst case constant $r = \frac{\mu}{\nu}$.

\begin{figure*}[t!]\label{fig:bench_strong_rule}
  \centering
  \subfigure{\includegraphics[width=0.49\columnwidth]{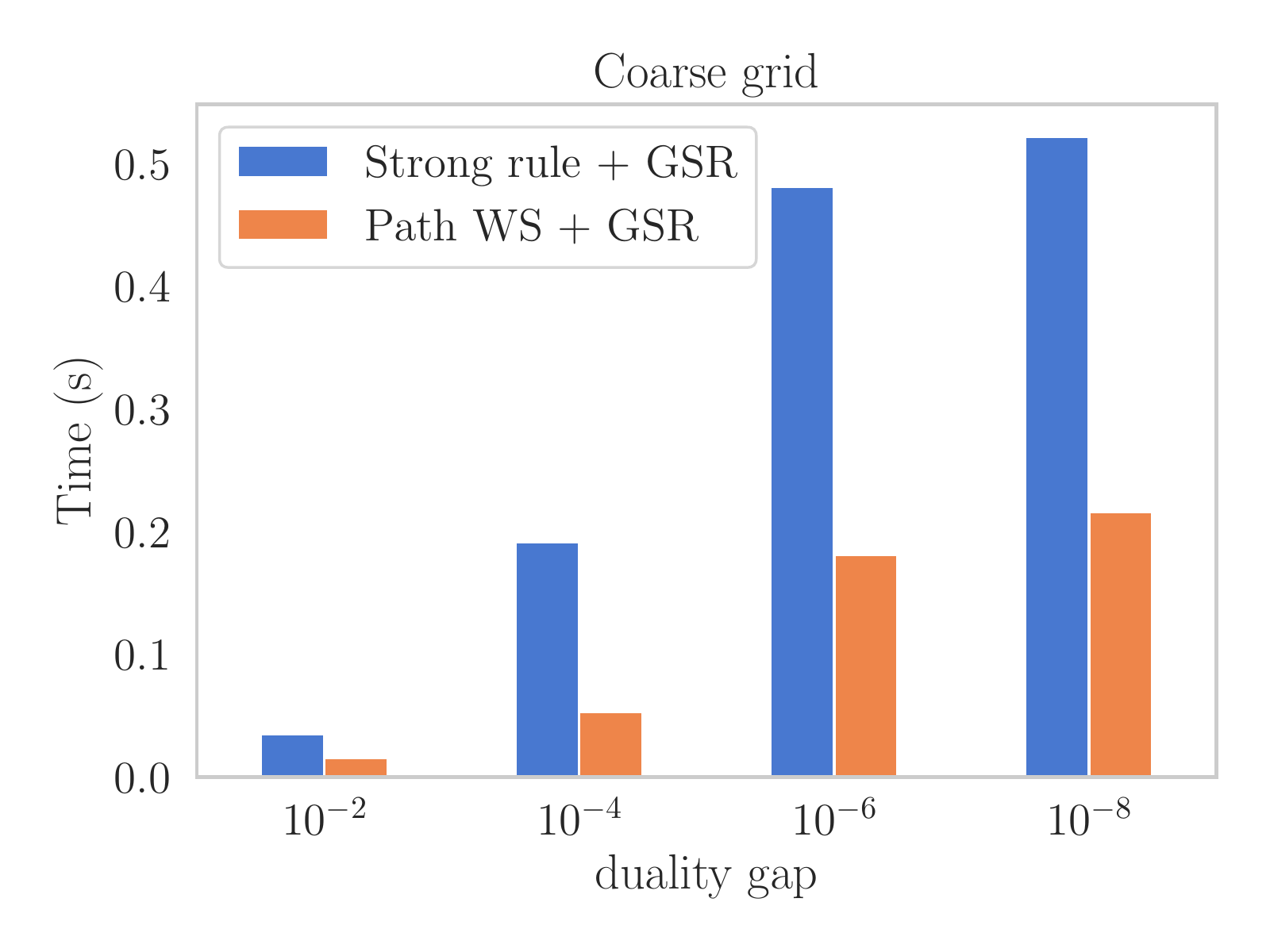}}
  \subfigure{\includegraphics[width=0.49\columnwidth]{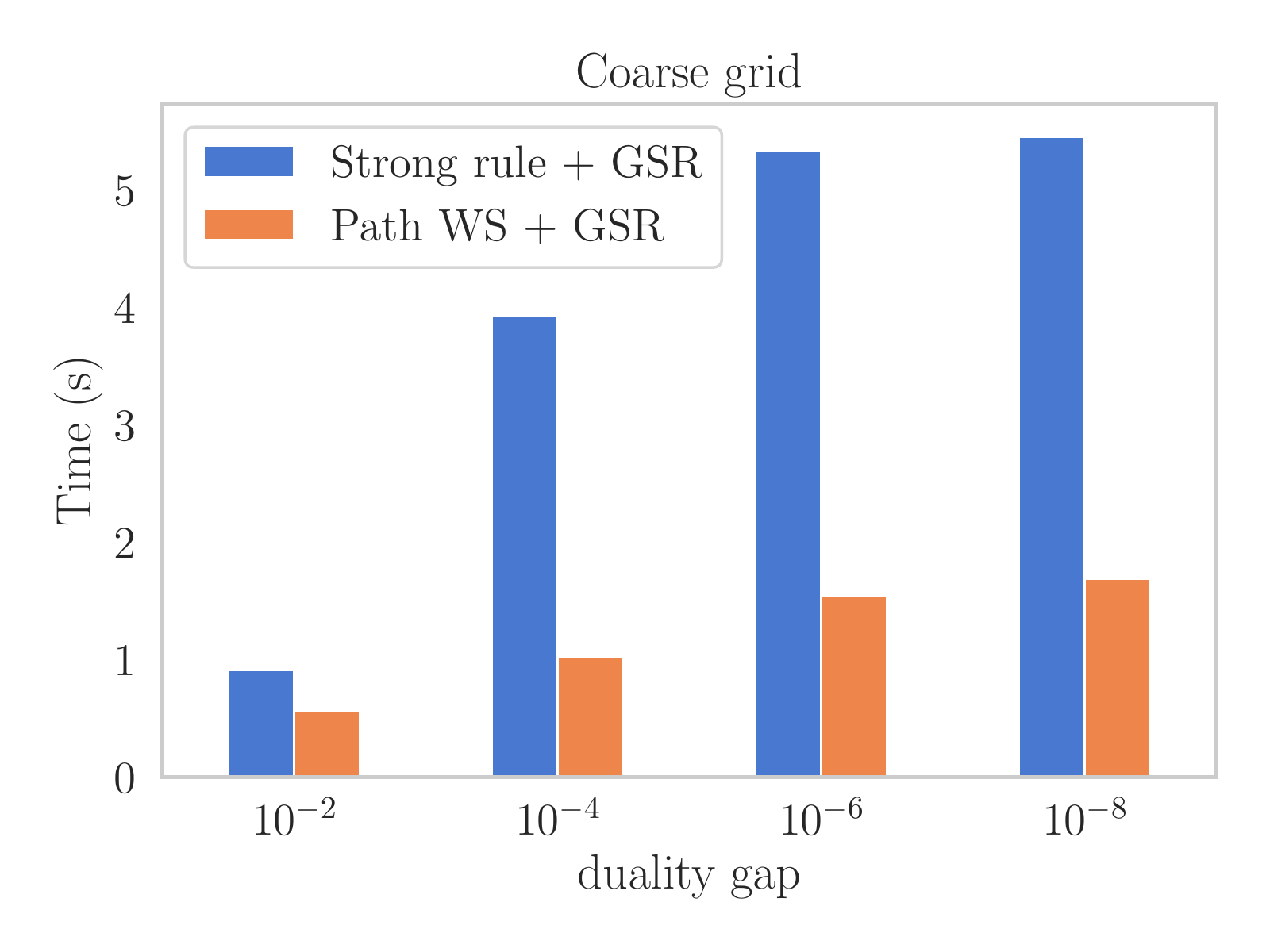}}
  \subfigure{\includegraphics[width=0.49\columnwidth]{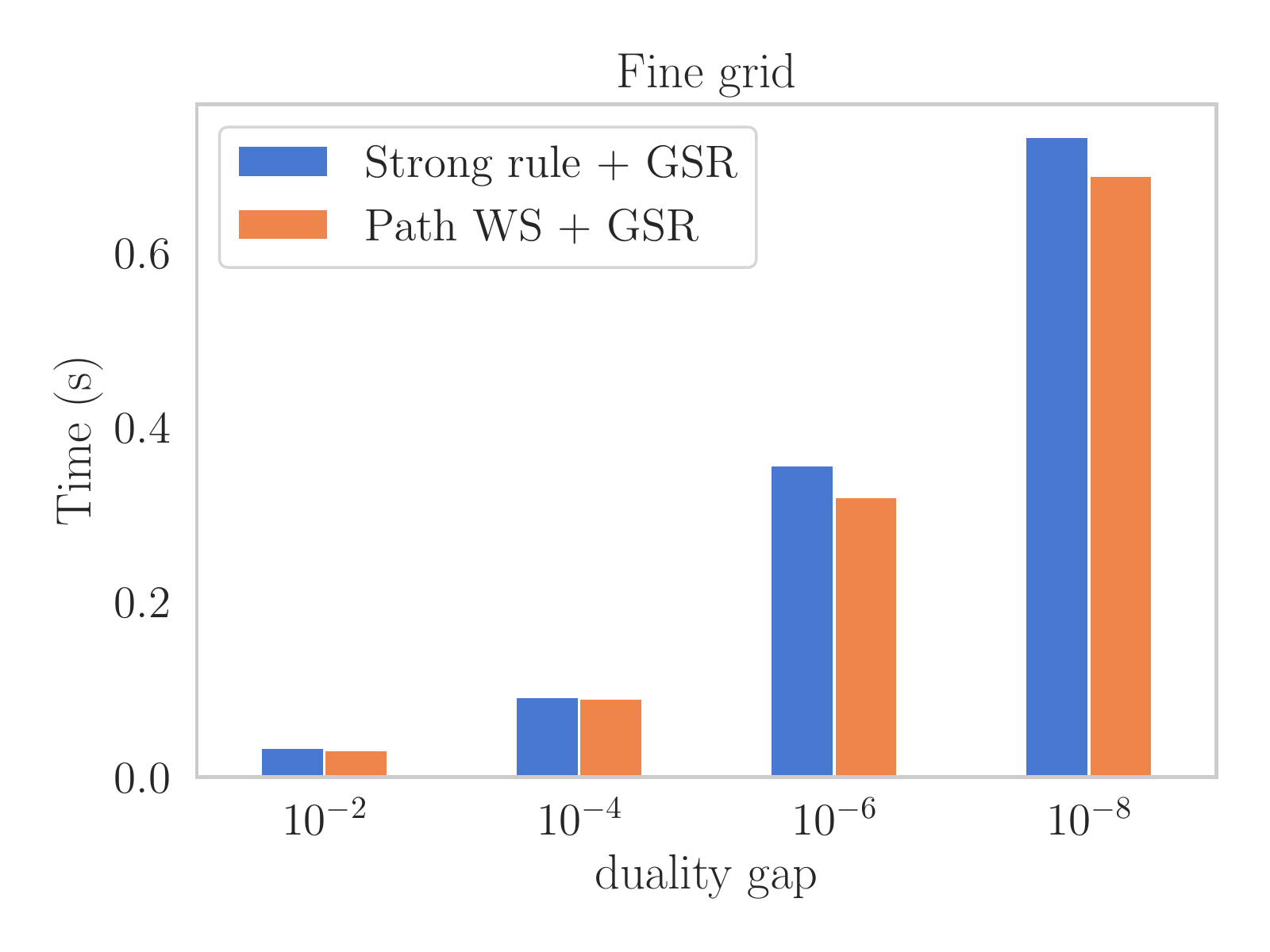}}
  \subfigure{\includegraphics[width=0.49\columnwidth]{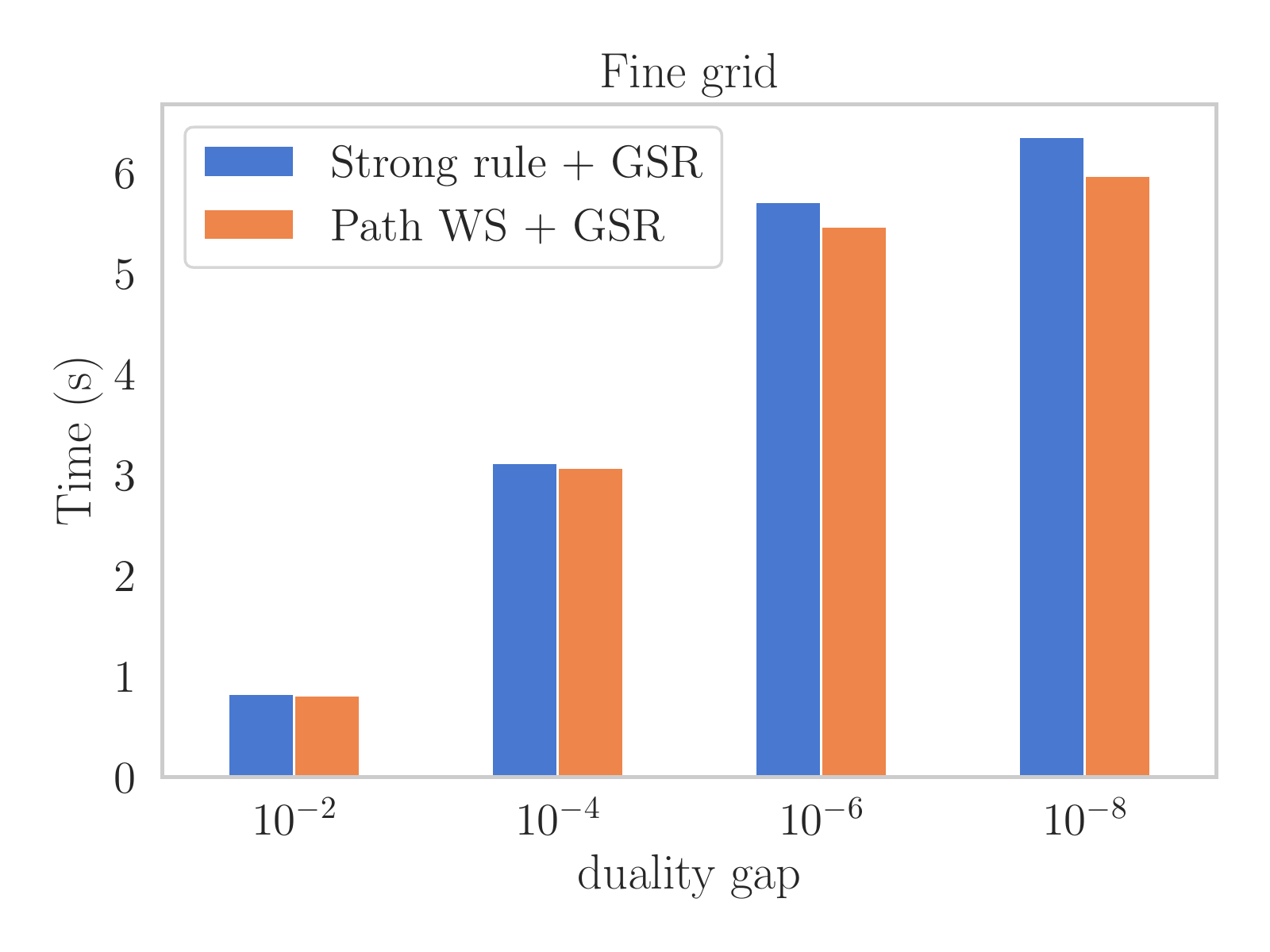}}
  \subfigure{\includegraphics[width=0.49\columnwidth]{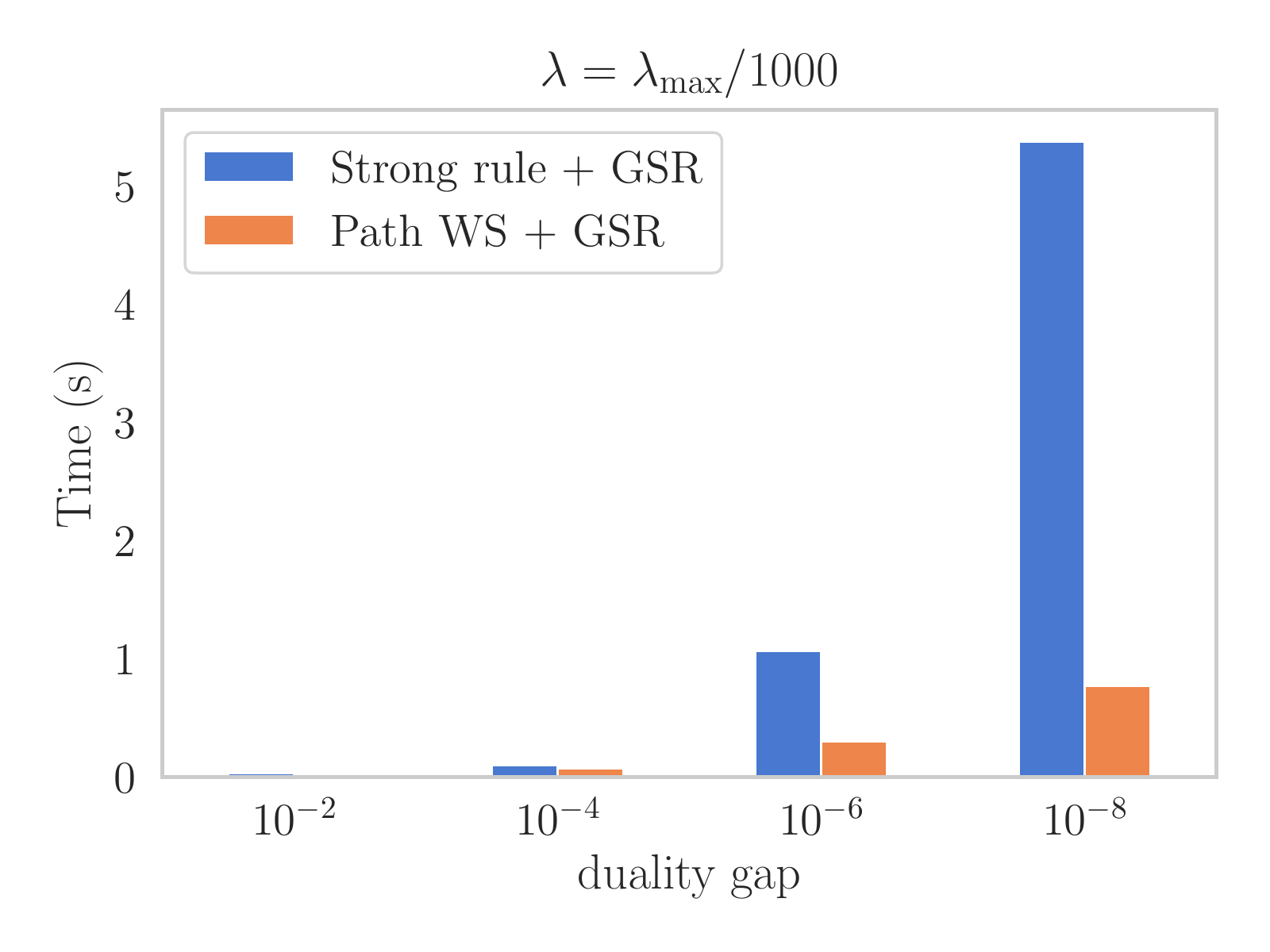}}
  \subfigure{\includegraphics[width=0.49\columnwidth]{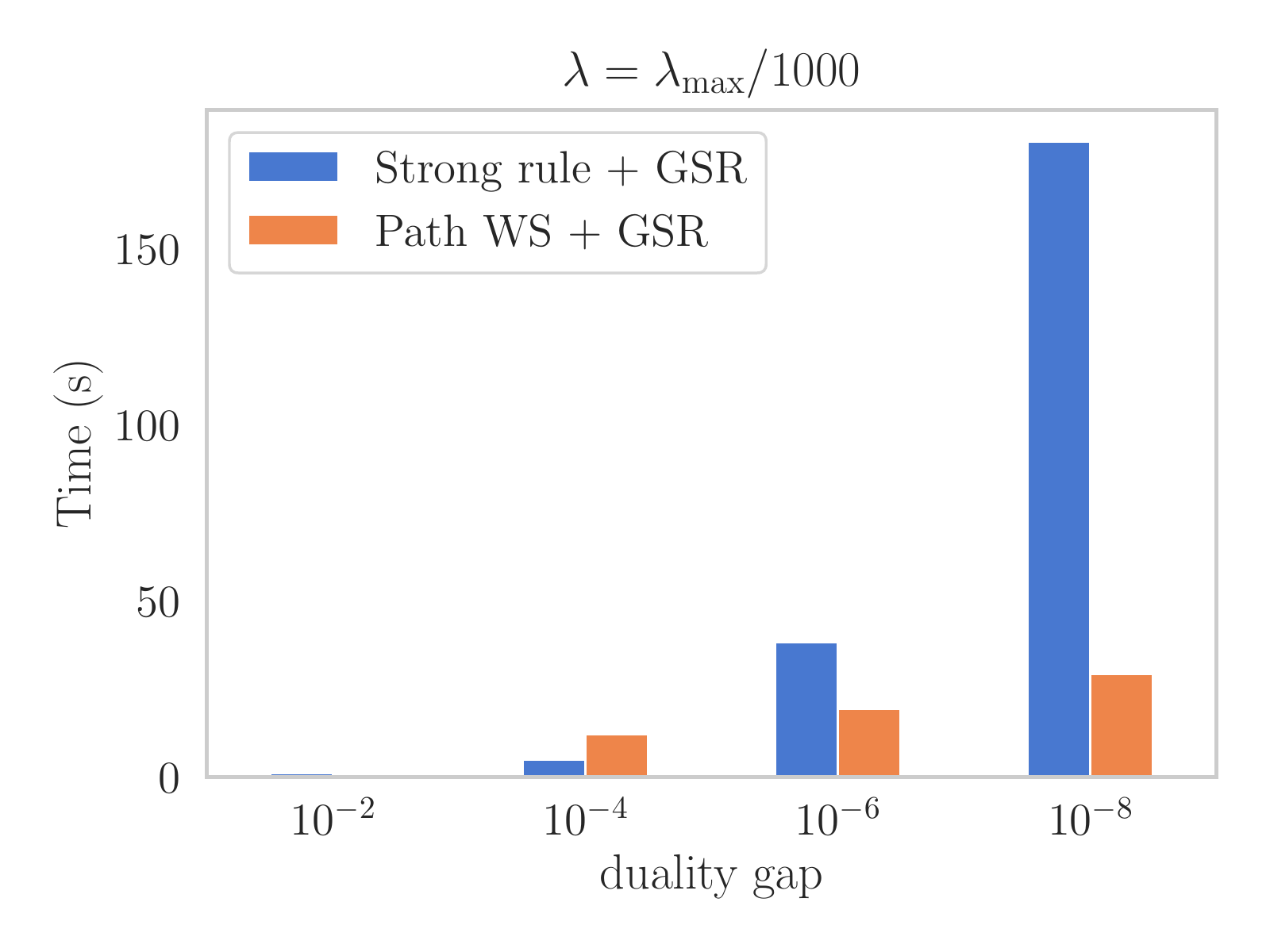}}
 \caption{Computation times needed to solve the Lasso regularization path of size $T=100$ (fine grid) and $T=10$ (coarse grid) starting from $\lambda_{\max}$ to $\lambda_{\max}/100$ at different accuracy. The benchmarks on the Left are launched on \texttt{Leukemia} dataset and the right on \texttt{Climate}.}
\end{figure*}

\subsection{Explicit Control of the Active Sets Sizes}

In this section, we analyze how the sequence of optimal support $\{j \in [p]: \tbeta{\lambda_t}_j \neq 0\}_t$ evolves. We will directly control the number of variable that will enters the active set at the next iteration. To do so, we first quantify the moment where a non active variable becomes active. Under a constraints on the size of the active set, we find the next regularization parameter $\lambda_{t+1}$ as far as possible from $\lambda_t$.


\begin{proposition}
[Path-Following with constrained active set size]
\label{prop:path_constraint_size}
Let $f$ satisfies assumption $(A)$. For any $\lambda' \in [\lambda, \lambda_t]$, it holds:

\begin{enumerate}
\item For any $j$ such that $d_{j}^{2}(\theta_t) > 2\nu \mathcal{E}'_{t}$, we have
$$\text{if }\lambda' \leq \frac{\lambda_t \normin{\theta_t}}{\normin{\theta_t} + \sqrt{d_{j}^{2}(\theta_t) - 2\nu \mathcal{E}'_{t}} } \text{ then } j \in \mathcal{A}_{t}(\lambda')\enspace. $$
\item For any $j$ such that $d_{j}^{2}(\theta_t) > 2\nu \mathcal{E}_{t}^{"}$, we have
$$ \text{if } j \in \mathcal{A}_{t}(\lambda'), \text{ then } \lambda' \leq \frac{\lambda_t \normin{\theta_t}}{\normin{\theta_t} + \sqrt{\frac{\mu}{\nu}\left(d_{j}^{2}(\theta_t)  - 2\nu \mathcal{E}_{t}^{"}\right)}} \enspace.$$ 
\end{enumerate}
\end{proposition}

We can now easily deduce the following control on the active set size which is larger or equal to the cardinality of the support. 

\begin{corollary}
Given the ordered sequence $d_{(1)}(\theta_t) \leq \cdots \leq d_{(p)}(\theta_t)$ and an integer $p_t$ in $[|\mathcal{A}_t|, p]$, it holds
\begin{align*}
&\lambda_{t+1} > \frac{\lambda_t \normin{\theta_t}}{\normin{\theta_t} + \sqrt{\frac{\mu}{\nu}(d_{(p_t)}^{2}(\theta_t) - 2\nu \mathcal{E}_{t}^{"})}} \Longrightarrow |\mathcal{A}_{t}(\lambda_{t+1})| < p_t \enspace,\\
&\lambda_{t+1} \leq \frac{\lambda_t \normin{\theta_t}}{\normin{\theta_t} + \sqrt{d_{(p_t)}^{2}(\theta_t) - 2\nu \mathcal{E}'_{t}} } \Longrightarrow |\mathcal{A}_{t}(\lambda_{t+1})| \geq p_t \enspace.
\end{align*}
Since $\hat{\mathcal{A}}^{(\lambda_{t+1})} \subset \mathcal{A}_{t}(\lambda_{t+1})$, we also have $|\hat{\mathcal{A}}^{(\lambda_{t+1})}| < p_t$.
\end{corollary}

This result provide an explicit and computable correspondence between the number of variables in the active set and the regularization hyperparameter at each step. In the special case of Lasso with exact solution, we have
$$ \lambda_{t+1} = \frac{\lambda_t \normin{\hat\theta_t}}{\normin{\hat\theta_t} + d_{(p_t)}(\hat\theta_t)} \Longleftrightarrow |\mathcal{A}_t(\lambda_{t+1})| = p_t \enspace.$$
With this pathwise screening point of view, we explicitly connect the regularization path strategies with working set selection heuristic that threshold the correlation between the dual variable and the features vector $\mathcal{W}_t = \{j \in[p]:\, d_j(\theta_t) \leq d_{(p_t)}(\theta_t)\}$.
Such working set policy can be interpreted as an unsafe approximation of the safe active set path by dropping the optimization error terms.\\

We can notice that the path with precise control on the active set size suggests a geometric grid, thus corroborating current heuristics. Moreover, such a pathwise following methodologies are analogue to Lars-Lasso algorithm \cite{Efron_Hastie_Johnstone_Tibshirani04} without the expensive Gram matrix inversion at every step. In large scale setting, it is usually favorable to use a first order algorithms to quickly approximate the primal dual optimal solutions; in which case the exact Lars-Lasso step sizes are not available. Therefore, our approach is an analogue of the safe approximation path in \cite{Ndiaye_Le_Fercoq_Salmon_Takeuchi2018} that is tailored on the active set size instead of the optimization error. As a direct consequence, maintaining an active set size increment of one, results in a Lars-Lasso algorithm involving only approximate solution.

\subsection{Pathwise Working Set}

By definition, the safe screening rule presented above have to be conservative enough to avoid false exclusion of any feature. A careful relaxation of the \emph{safety} constraint allows to eliminate more variables and presumably one obtain faster problem solving abilities. In that sense, the \texttt{strong rule} \cite{Tibshirani_Bien_Friedman_Hastie_Simon_Tibshirani12} assumes that the gradient of the loss function is $1$-Lipschitz continuous along the regularization path and removes variable $j$ at the $t+1$ stage when
$$ |X_{j}^{\top} \ttheta{\lambda_{t}}| < \frac{2 \lambda_{t+1} - \lambda_t}{\lambda_t} \enspace.$$
Unfortunately, such assumptions do not hold without stronger assumptions on the design matrix $X$. Whence a correction step must accompany the \texttt{strong rule} strategies. Similarly, recent \emph{working set} approaches \citep{Jonhson_Guestrin15, Massias_Gramfort_Salmon18} prioritize the features by following similar strategy than safe screening methods but without a conservative optimization error term in order to aggressively removes more variables. They solve a sequence of subproblem starting with a small subset of features and iteratively includes more variables until convergence. Each subproblem is launched at the same target regularization parameter $\lambda$. Lately, these strategies have led to considerable improvements in the resolution of Lasso type problem both in classification and regression. These strategies can be harmoniously combined to build a fast solver by taking benefits of both sequence of screening/working sets and sequence of decreasing regularization parameters. Note that
a variable $j$ is removed when 
$$|X_{j}^\top \ttheta{\lambda}| = |X_{j}^\top \nabla f(X\tbeta{\lambda})| < \lambda \enspace,$$ which suggest a straight replacement of $\nabla f(X\tbeta{\lambda_{t+1}})$ by the current estimates $\nabla f(X\beta_t)$, we obtain an optimistic unsafe estimation of the optimal active set $\hat{\mathcal{A}}^{(\lambda)}$ with $\mathcal{W}(\beta_t, \lambda_{t+1})$ where
%
for $\beta \in \bbR^p$ and $\gamma > 0$,
\begin{equation}
\mathcal{W}(\beta, \gamma) = \{j \in [p]:\, |X_{j}^{\top}\nabla f(X\beta)| \geq \gamma\}\enspace.
\end{equation}
Therefore, the cardinality of the working set is directly controlled by the policies for selecting $\lambda_{t+1}$ given the previous parameter $\lambda_t$. Contrarily to the \texttt{strong rule} which is unusable when the step size between two $\lambda$s is larger than $2$, our simpler rule allows more flexibility and one can use the pathwise policies as a rule for active set changes. This is illustrated in \Cref{fig:bench_strong_rule} where a significant gain in computation time is obtained when the regularization path grid is coarse.

We believe that such methods can also be used in \cite{Larsson_Wallin21}'s approach that combines the exploitation of solutions along the path with second order information to obtain better initialization.

\section{Discussions}


Approximate continuation path is used in many optimization solvers widely used in the machine learning community (see \texttt{scikit-learn}\cite{Pedregosa_etal11} or \texttt{glmnet} \cite{Friedman_Hastie_Tibshirani10}). While, in some setting, it offers great numerical performance in terms of computational time, its theoretical properties is still not fully understood. We are not aware of any path design that systematically leads to a significant speed up without additional trick. In this paper, we relied on a theoretical analysis of the rate of convergence of the duality gap at a prescribed parameter $\lambda$ during a pathwise optimization. We show that one can guarantee a linear convergence on the target gap. However, this alone, does not necessarily translate into faster algorithm in all situations.
With additional simplifications to avoid some worst case numerical issues and unnecessary conservativeness, we presented simple variants of approximate continuation path that are easier to implement.\\

For sparse optimization problem, we provided explicit bounds on the size of the active sets along the regularization path. As a consequence, it allows to extend the Lars-Lasso algorithm when iterative solvers are used instead of computing exact solutions. The current successful speed up strategies maintains low iteration complexity when solving sparse problem thanks to two key ingredients. First, generates an adaptive sequence of regularization parameter that makes a fast progress on the target problem or precisely control the size of the sequential safe active set when the regularization induces sparsity. Second, generates a simple sequence of working set along with a careful early stopping criterion that saves computational effort. Both proposed principles for designing a pathwise optimization process lead to a geometric grid which not only provide formal validation of the actual practices, but substantially improves pathwise screening rule solvers when a coarse grid are used.\\

Despite its widespread use, homotopy continuation remains a vaguely justified heuristic that does not necessarily accelerate iterative solvers for generic composite optimization. Even in convex settings, further theoretical investigations will be valuable.

\newpage
\section*{Acknowledgements}
This work was partially supported by MEXT KAKENHI (20H00601, 16H06538), JST CREST (JPMJCR21D3), JST Moonshot R\&D (JPMJMS2033-05), NEDO (JPNP18002, JPNP20006) and RIKEN Center for Advanced Intelligence Project.

\bibliography{references}
\bibliographystyle{plain}

\newpage

\section{Appendix}

In this section, we complete the proofs of the article's propositions.



We start with the following classical result.
\begin{lemma}\label{lm:bound_gradient_objective} If $f$ satisfies assumption $(A)$, we have
\begin{equation*}
\frac{1}{2\nu} \norm{\nabla f(x)}^2 \leq f(x) \leq \frac{1}{2\mu} \norm{\nabla f(x)}^2 \enspace.
\end{equation*}
\end{lemma}

\begin{proof}
From the smoothness of $f$, we have:
\begin{equation*}
\inf_z f(z) \leq \inf_z (f(x) + \langle \nabla f(x), z - x \rangle + \frac{\nu}{2}\norm{z - x}^2) = f(x) - \frac{1}{2\nu}\norm{\nabla f(x)}^2 \enspace.
\end{equation*}
The second inequality follows from the same techniques.
\end{proof}

\begin{lemma}[c.f. \Cref{lm:approximate_stepwise_progress}]
We suppose that the function $f$ satisfies assumption $(A)$ and that the monotonicity condition $f(X\beta_{t+1}) \leq f(X\beta_t)$ holds. Then
\begin{align*}
\Gap_{\lambda}(\beta_{t+1}, \theta_{t+1}) - \Gap_{\lambda}(\beta_{t}, \theta_{t}) \leq \mathcal{E}_{t+1} - \mathcal{E}_t - \frac{\delta_t \normin{\zeta_t}_{2}^{2}}{2 \nu} \enspace,
\end{align*}
where
$$\delta_t := \left(1-\frac{\lambda}{\lambda_{t}}\right)^2 - \left(\frac{\alpha_{t}\nu}{\lambda_{t}\mu}\right)^2 \left(1-\frac{\lambda}{\lambda_{t+1}}\right)^2 \enspace.$$
\end{lemma}

\begin{proof}
For simplicity, we will temporarily denote $G_t = \Gap_{\lambda}(\beta_{t}, \theta_{t})$.

From \Cref{eq:warmstart_bound}, we have
\begin{align*}
G_{t+1} &\leq \mathcal{E}_{t+1} + \frac{1}{2\mu}\norm{\zeta_{t+1}}^2 \left(1 - \frac{\lambda}{\lambda_{t+1}} \right)^2\\
G_{t} &\geq \mathcal{E}_{t} + \frac{1}{2\nu}\norm{\zeta_t}^2 \left(1 - \frac{\lambda}{\lambda_{t}} \right)^2 \enspace.
\end{align*}
This implies that
\begin{equation}\label{eq:stepwise_diff_ineq}
G_{t+1} - G_{t} \leq \mathcal{E}_{t+1} - \mathcal{E}_{t} + \frac{1}{2\mu}\norm{\zeta_{t+1}}^2 \left(1 - \frac{\lambda}{\lambda_{t+1}} \right)^2 - \frac{1}{2\nu}\norm{\zeta_t}^2 \left(1 - \frac{\lambda}{\lambda_{t}} \right)^2 \enspace.
\end{equation}
From \Cref{lm:bound_gradient_objective} and the pathwise monotonicity assumption on $f$, we have
\begin{equation*}
\frac{1}{2\nu}\norm{\nabla f(X\beta_{t+1})}^2 \leq f(X\beta_{t+1}) \leq f(X\beta_{t}) \leq 
\frac{1}{2\mu}\norm{\nabla f(X\beta_{t})}^2 \enspace.
\end{equation*}
This implies (where we also use $\alpha_{t+1} \geq \lambda_{t+1}$ in the first inequality)
\begin{equation}\label{eq:decrease_residual}
\norm{\zeta_{t+1}}^2 = \left( \frac{\lambda_{t+1}}{\alpha_{t+1}} \norm{\nabla f(X\beta_{t+1})}\right)^2 \leq \frac{\nu}{\mu} \norm{\nabla f(X\beta_{t})}^2 = \frac{\nu}{\mu} \left(\frac{\alpha_t}{\lambda_t}\right)^2 \norm{\zeta_t}^2 \enspace.
\end{equation}
By using \Cref{eq:decrease_residual} into \Cref{eq:stepwise_diff_ineq}, we obtain
\begin{equation*}
G_{t+1} - G_{t} \leq \mathcal{E}_{t+1} - \mathcal{E}_{t} + \frac{1}{2\mu}\frac{\nu}{\mu} \left(\frac{\alpha_t}{\lambda_t}\right)^2 \norm{\zeta_t}^2 \left(1 - \frac{\lambda}{\lambda_{t+1}} \right)^2 - \frac{1}{2\nu}\norm{\zeta_t}^2 \left(1 - \frac{\lambda}{\lambda_{t}} \right)^2 \enspace.
\end{equation*}
The result follows from a factorization of the right hand side.
\end{proof}

\begin{remark}[Variation of the loss function along the path]
Given an algorithm for solving \Cref{eq:primal}, the assumption $f(X\beta_{t+1}) \leq f(X\beta_t)$ might not hold for every iterations. However it always holds at optimality \ie $\lambda \mapsto f(X\tbeta{\lambda})$ is non-increasing \cite[Lemma 7]{Ndiaye_Le_Fercoq_Salmon_Takeuchi2018}. One can either include an additional optimization error term or ensure that the bound holds before stopping the algorithm at time $t+1$. We have considered the latter for simplicity of the analysis.
\end{remark}

\begin{proposition}[c.f. \Cref{prop:fast_path}]
We assume that $\mathrm{(A)}$ holds. For any $r \in [0, \frac{\mu}{\nu}]$, $T \geq 1$, we assume for any $t \in [T]$, $f(X\beta_{t+1}) \leq f(X\beta_t)$ and $\mathcal{E}_{t+1} \leq (1 - r) \mathcal{E}_t + \epsilon_t$.
By selecting $\lambda_{t+1}$ such that
\begin{equation*}
\left(\frac{\alpha_{t}\nu}{\lambda_{t}\mu}\right)^2 \left(1 - \frac{\lambda}{\lambda_{t+1}} \right)^2 \leq \left(1 - r \frac{\nu}{\mu}\right) \left(1 - \frac{\lambda}{\lambda_{t}}\right)^2 - \frac{2 \nu \epsilon_t}{\norm{\zeta_t}^2}\enspace,
\end{equation*}
we have
\begin{equation*}
\Gap_{\lambda}(\beta_T, \theta_T) \leq (1 - r)^T \Gap_{\lambda}(\beta_0, \theta_0) \enspace.
\end{equation*}
\end{proposition}

\begin{proof}
Let us denote $G_t = \Gap_{\lambda}(\beta_{t}, \theta_{t})$. By assumption and definition of $\delta_t$, we have
\begin{align*}
\frac{1}{2\nu}\norm{\zeta_t}^2 \delta_t \geq \epsilon_t  + \frac{r}{2\mu}\norm{\zeta_t}^2\left(1 - \frac{\lambda}{\lambda_t} \right)^2 \text{ and } \epsilon_t \geq (r-1) \mathcal{E}_{t} + \mathcal{E}_{t+1} \enspace.
\end{align*}
This implies that
\begin{align*}
&\qquad \frac{1}{2\nu}\norm{\zeta_t}^2 \delta_t \geq (r-1) \mathcal{E}_{t} + \mathcal{E}_{t+1}  + \frac{r}{2\mu}\norm{\zeta_t}^2\left(1 - \frac{\lambda}{\lambda_t} \right)^2 \\
&\Longleftrightarrow
\mathcal{E}_{t+1} - \mathcal{E}_{t} - \frac{1}{2\nu}\norm{\zeta_t}^2 \delta_t \leq -r \left[\mathcal{E}_t + \frac{1}{2\mu}\norm{\zeta_t}^2\left(1 - \frac{\lambda}{\lambda_t} \right)^2\right]\\
&\Longrightarrow
\mathcal{E}_{t+1} - \mathcal{E}_{t} - \frac{1}{2\nu}\norm{\zeta_t}^2 \delta_t \leq -r G_t \qquad \text{ (upper bound in \Cref{eq:warmstart_bound}) }\\
&\Longrightarrow G_{t+1} - G_t \leq - r G_t \enspace,
\end{align*}
where the last inequality comes from \Cref{lm:approximate_stepwise_progress}. Whence
$$ G_T \leq (1 - r) G_{T-1} \leq (1 - r)^2 G_{T-2} \leq \cdots \leq (1 - r)^T G_{0} \enspace.$$ 
\end{proof}

\begin{proposition}[c.f. \Cref{prop:path_constraint_size}]
Let $f$ satisfies assumption $(A)$. For any $\lambda' \in [\lambda, \lambda_t]$, it holds:

\begin{enumerate}
\item For any $j$ such that $d_{j}^{2}(\theta_t) > 2\nu \mathcal{E}'_{t}$, we have
$$\text{if }\lambda' \leq \frac{\lambda_t \normin{\theta_t}}{\normin{\theta_t} + \sqrt{d_{j}^{2}(\theta_t) - 2\nu \mathcal{E}'_{t}} } \text{ then } j \in \mathcal{A}_{t}(\lambda')\enspace. $$
\item For any $j$ such that $d_{j}^{2}(\theta_t) > 2\nu \mathcal{E}_{t}^{"}$, we have
$$ \text{if } j \in \mathcal{A}_{t}(\lambda'), \text{ then } \lambda' \leq \frac{\lambda_t \normin{\theta_t}}{\normin{\theta_t} + \sqrt{\frac{\mu}{\nu}\left(d_{j}^{2}(\theta_t)  - 2\nu \mathcal{E}_{t}^{"}\right)}} \enspace.$$ 
\end{enumerate}
\end{proposition}


\begin{proof}
We remind that $r_t(\lambda) = \sqrt{\frac{2\nu}{\lambda^2} \Gap_{\lambda}(\beta_t, \theta_t)}$ and $\norm{\zeta_t} = \lambda_t \norm{\theta_t}$. From \Cref{eq:warmstart_bound}, we have
\begin{align*}
%
\norm{\zeta_t}^2 \left(\frac{1}{\lambda} - \frac{1}{\lambda_t} \right)^2 &\leq
r_{t}^{2}(\lambda) - \frac{2\nu}{\lambda^2} \mathcal{E}_t(\lambda) \leq \frac{\nu}{\mu} \norm{\zeta_t}^2 \left(\frac{1}{\lambda} - \frac{1}{\lambda_t} \right)^2 \enspace.
\end{align*}
A coordinate $j \in [p]$ is in the safe active set at parameter $\lambda$ \ie $j \in \mathcal{A}_t(\lambda) \Longleftrightarrow r_t(\lambda) \geq d_j(\theta_t)$
\begin{align*}
\lambda \leq \frac{\lambda_t \normin{\theta_t}}{\normin{\theta_t} + \sqrt{d_{j}^{2}(\theta_t) - 2\nu \mathcal{E}'_{t}} } &\Longrightarrow
\frac{2\nu}{\lambda^2} \mathcal{E}_t(\lambda) +
\norm{\zeta_t}^2 \left(\frac{1}{\lambda} - \frac{1}{\lambda_t} \right)^2 \geq d_{j}^{2}(\theta_t) \\
&\Longrightarrow r_t(\lambda) \geq d_j(\theta_t) \enspace.
\end{align*}

A coordinate $j \in [p]$ is not in the safe active set at parameter $\lambda$ \ie $j \notin \mathcal{A}_t(\lambda) \Longleftrightarrow r_t(\lambda) < d_j(\theta_t)$
\begin{align*}
\lambda > \frac{\lambda_t \normin{\theta_t}}{\normin{\theta_t} + \sqrt{\frac{\mu}{\nu}} \sqrt{d_{j}^{2}(\theta_t) - 2 \nu \mathcal{E}_{t}^{"}} } &\Longrightarrow
\frac{2\nu}{\lambda^2} \mathcal{E}_t(\lambda) +
\frac{\nu}{\mu} \norm{\zeta_t}^2 \left(\frac{1}{\lambda} - \frac{1}{\lambda_t} \right)^2 < d_{j}^{2}(\theta_t) \\
&\Longrightarrow r_t(\lambda) < d_j(\theta_t) \enspace.
\end{align*}
\end{proof}

\end{document}